\documentclass[12pt,reqno]{amsart}
\usepackage{amssymb,amsthm,amscd}
\usepackage[utf8]{inputenc}
\usepackage{ytableau}
\usepackage{graphicx}

\newtheorem{theorem}{Theorem}
\newtheorem{proposition}{Proposition}
\newtheorem{lemma}{Lemma}
\newtheorem{corollary}{Corollary}

\theoremstyle{definition}

\newtheorem{definition}{Definition}

\def\Var{\mathrm{Var}}

\def\di{\mathrm{di}}

\def\Perm{\mathrm{Perm}}

\def\SLC{\mathrm{SLC}}
\def\Nov{\mathrm{Nov}}
\def\Der{\mathrm{Der}}

\def\wt{\mathop{\fam 0 wt}\nolimits}
\def\Magma{\mathrm{Magma}}
\let\<\langle
\let\>\rangle

\numberwithin{equation}{section}

\title[Novikov dialgebras and perm algebras]{Novikov dialgebras and perm algebras}

\author{F. Mashurov, B. K. Sartayev}

\thanks{This work was supported by the Program of MES RK}

\address{Sobolev Institute of Mathematics, Novosibirsk, Russia}

\address{Suleyman Demirel University, Kaskelen, Kazakhstan}


\begin{document}

\begin{abstract}
In this paper, we consider Perm algebra with the
derivation $d$. The algebra itself is equipped with the new operation $a\succ b = d(a) b$.
We construct a linear basis of the free Novikov dialgebra in terms of new operations. Also, we prove that the class of algebras under the new operation form a variety.
Finally, we find the defining identities of the variety.
\end{abstract}

\maketitle

\section{Introduction}

An associative algebra with the left-commutative identity is called perm algebra \cite{Perm 1}. In \cite{MS2022} subalgebras of the perm algebra under the commutator and the anti-commutator are considered. It turned out that every metabelian Lie algebra can be embedded into perm algebra under the commutator.
In \cite{KS2022} was considered a subalgebra of perm algebra with derivation under the operation $x_1\prec x_2=x_1 d(x_2)$. We denote the perm algebra with derivation and its subalgebra under the operation $\prec$ by $\Der\Perm\<X\>$ and $\Der\Perm_{\prec}\<X\>$, respectively. It turn out that $\Der\Perm_{\prec}\<X\>$ is a pre-Lie algebra with additional identities of the following form:
\begin{gather}
    ((x_1\prec x_2)\prec x_3)\prec x_4 = ((x_1\prec x_3)\prec x_2)\prec x_4 ,\label{eq:SLS-1} \\
    (x_1,x_2\prec x_3, x_4) = (x_2,x_1\prec x_3, x_4). \label{eq:SLS-2}
\end{gather}
where $(x_1, x_2, x_3)$ stands for the associator $(x_1\prec x_2)\prec x_3-x_1\prec(x_2\prec x_3)$. Moreover, there is proved that every Novikov dialgebra can be embedded into perm algebra with derivation as follows:
$$x_1\vdash x_2 = x_1 d(x_2),\;\;\; x_1\dashv x_2 = d(x_2) x_1.$$

In this paper, we consider a new operation $\succ$ on $\Der\Perm\<X\>$ algebra which is defined as follows:
$$x_1\succ x_2 = d(x_1) x_2.$$
We denote that algebra by $\Der\Perm_{\succ}\<X\>$.
In our case the operations $\vdash$ and $\dashv$ correspond to the operations $\prec$ and $\succ$, respectively. The main aim of the paper is to describe the subalgebra of $\Der\Perm\<X\>$ under the operation $\succ$.

Let us call that an algebra $A$ is special if it can be embedded into some algebra from fixed class of algebras, otherwise, it is called exceptional. It was proved that there exists pre-Lie algebra with additional identities (\ref{eq:SLS-1}) and (\ref{eq:SLS-2}) that is exceptional i.e. can not be embedded into the class of $\Der\Perm$ 
algebras.  For the case of the class of $\Der\Perm_{\succ}$ algebras, we have the opposite result. In addition, it turns out that the class of $\Der\Perm_{\succ}$ algebras forms a variety. Moreover, we describe the defining identities of that variety and construct a monomial basis of free $\Der\Perm_{\succ}$ algebra. The analogical result can be traced in the class of subalgebras of commutative and Poisson algebras with derivation \cite{Bokut2}, \cite{KS22expmath}.

The main motivation of this work comes from the theory of operads. It is well-known that for the fixed binary quadratic operad denoted by $\Var$, we can define the Manin white product with operad $\Nov$ governed by the variety of Novikov algebras and obtain the operad $\Der\Var$ \cite{KSO2019}. That operad is obtained from operad $\Var$ under the operations $\succ$ and $\prec$ which are defined as follows:
$$x_1\succ x_2=d(x_1) x_2,\;\;\;x_1\prec x_2=x_1 d(x_2).$$
In the other words, we have
$$\Var\circ\Nov=\Der\Var.$$
On the other hand \cite{Loday2001}, \cite{Vallette},
$$\Perm\circ\Var=\di\textrm{-}\Var.$$
In this paper and \cite{KS2022}, a special case is considered. That is
$$\Perm\circ\Nov=\di\textrm{-}\Nov=\Der\Perm.$$

Let us now pay close attention to Novikov dialgebra.
Every dialgebra
satisfies to the {\em 0-identities}:
\begin{equation}\label{eq:zero-id}
\begin{gathered}
(x_1\dashv x_2)\vdash x_3 = (x_1\vdash x_2)\vdash x_3,
\\
x_1\dashv (x_2\vdash x_3) = x_1\dashv (x_2\dashv x_3) .
\end{gathered}
\end{equation}

\begin{definition}
The Novikov dialgebra is defined by \eqref{eq:zero-id} together with
\begin{equation}\label{eq:diLSym}
 \begin{gathered}
(x_1\vdash x_2)\vdash x_3 - x_1\vdash (x_2\vdash x_3)
= (x_2\vdash x_1)\vdash x_3 - x_2\vdash (x_1\vdash x_3),\\
(x_1\dashv x_2)\dashv x_3 = (x_1\dashv x_3)\dashv x_2,
 \end{gathered}
\end{equation}
\begin{equation}\label{eq:diRCom}
 \begin{gathered}
(x_1\dashv x_2)\dashv x_3 - x_1\dashv (x_2\dashv x_3)
= (x_2\vdash x_1)\dashv x_3 - x_2\vdash (x_1\dashv x_3),\\
(x_1\vdash x_2)\dashv x_3 = (x_1\vdash x_3)\vdash x_2.
 \end{gathered}
\end{equation}
\end{definition}

We see that the Novikov dialgebra is left-symmetric and right-commutative under the operations $\vdash$ and $\dashv$, respectively. Since a free basis for these algebras has been constructed (for example, see \cite{Bokut1}), it is not difficult to construct a basis of the free algebra with identities (\ref{eq:diLSym}). At times it is difficult to solve the same problem for algebras with complementary identities that contains both multiplications as (\ref{eq:zero-id}) and (\ref{eq:diRCom}). For the free Novikov dialgebras, we solve this problem in terms of $\succ$ and $\prec$.

\section{Main results}
Firstly, let us recall a free basis of $\Perm\<X\>$ algebra. That is monomials of the following form:
\begin{equation}\label{permmonom}
x_{i_1}\cdots x_{i_{n-1}}x_{i_n},
\end{equation}
where $i_1\leq\ldots\leq i_{n-1}$. Since the generators of $\Der\Perm\<X\>$ are
$$X^{(\omega)}=X\cup X'\cup X''\cdots\cup X^{(n)}\cup\cdots,$$
every monomial $a\in\Der\Perm\<X\>$ can define written as follows:
\begin{equation}\label{derpermmonom}
a=x_{i_1}^{(k_1)}\cdots x_{i_{n-1}}^{(k_{n-1})}x_{i_n}^{(k_n)}.
\end{equation}

To state the main results let us recall the crucial definition that used in \cite{DzhLofwall2002}, \cite{KSO2019}:
\begin{definition}
The weight function is defined on monomials of $\Der\Perm\<X\>$ as follows:
\[
\wt:\Der\Perm\<X\>\rightarrow \mathbb{Z},
\]
\[
\wt(x^{(j)}) =j-1,\quad x\in X,\ j\ge 0,
\]
\[
\wt(uv) = \wt(u)+\wt(v)
\]
for all monomials $u,v\in \Der\Perm\<X\>$.
\end{definition}

\begin{theorem}\label{criterion}
Let $a$ be a monomial of $\Der\Perm\<X\>$ of the form (\ref{derpermmonom}). A monomial $a\in\Der\Perm_{\succ}\<X\>$ if and only if $\wt(a)=-1$ and $k_n=0$.
\end{theorem}

Using Theorem \ref{criterion}, we obtain the dimension of operad $\Der\Perm_{\succ}$:
\begin{corollary}
  $\dim(\Der\Perm_{\succ}(n))=n\binom{2n-3}{n-1}$.
\end{corollary}

\begin{theorem}\label{directsum}
Every element of the free Novikov dialgebra can be written in a unique way as a sum of elements $\Der\Perm_{\succ}\<X\>$ and $\Der\Perm_{\prec}\<X\>$, that is, 
$$\di\textrm{-}\Nov\<X\>=\Der\Perm_{\succ}\<X\>\oplus\Der\Perm_{\prec}\<X\>.$$
as a vector space.
\end{theorem}

\begin{theorem}\label{Cohn}
Every homomorphic image of $\Der\Perm_{\succ}\<X\>$ is special.
\end{theorem}

Using Theorem \ref{Cohn}, we obtain the following result:

\begin{corollary}
  The class of $\Der\Perm_{\succ}\<X\>$ algebras form a variety.
\end{corollary}

Since the class of $\Der\Perm_{\succ}\<X\>$ algebras is a variety and thus defined by some set of identities, the interesting problem is to determine a list of defining identities of $\Der\Perm_{\succ}\<X\>$.
The next result answers that question:

\begin{theorem}\label{specialidentities}
The defining identities of the variety $\Der\Perm_{\succ}$ are 
  \begin{equation}\label{id1}
    x_1\succ(x_2\succ x_3)=x_2\succ(x_1\succ x_3),
  \end{equation}
    \begin{equation}\label{id2}
        \langle x_1,x_2\succ x_3,x_4\rangle-\langle x_1,x_3, x_2\succ x_4\rangle=0,
  \end{equation}
    \begin{equation}\label{id3}
       \langle x_1, x_2,x_3\rangle\succ x_4 -\langle x_1,x_3,x_2\rangle\succ x_4=0,
  \end{equation}
  where $\langle a, b,c \rangle= (a\succ b)\succ c-a\succ(b\succ c).$
  
\end{theorem}


\section{The proof of Theorem \ref{criterion} and Theorem \ref{directsum}}

\begin{proof}
It is clear that every monomial of $\Der\Perm_{\succ}\<X\>$ can be written as a sum of monomials of the form (\ref{derpermmonom}), where $\wt(a)=-1$ and $k_n=0$.

It remains to show that the monomial of the form (\ref{derpermmonom}) with conditions $\wt(a)=-1$ and $k_n=0$ can be written in terms of $\succ$. Let us prove it by induction on the length of monomial (\ref{derpermmonom}). The base of induction is $n=2$ and we have $x_1'x_2=x_1\succ x_2$.

For the monomial of the length $n$, we consider two cases depending on the appearance of the derivation degree $1$ in ${(k_1)},\cdots ,{(k_{n-1})}.$ If there is $r$ such that $k_{r}=1,$ then we have 
$$
x_{i_1}^{(k_1)}\cdots x_{i_{r-1}}^{(k_{r-1})} x_{i_{r}}' x_{i_{r+1}}^{(k_{r+1})}\cdots x_{i_{n-1}}^{(k_{n-1})}x_{i_n}. $$
By using left-commutative identity and definition of product $\succ$ we have 
$$
x_{i_1}^{(k_1)}\cdots x_{i_{r-1}}^{(k_{r-1})} x_{i_{r}}' x_{i_{r+1}}^{(k_{r+1})}\cdots x_{i_{n-1}}^{(k_{n-1})}x_{i_n}= 
x_{i_1}^{(k_1)}\cdots x_{i_{r-1}}^{(k_{r-1})}  x_{i_{r+1}}^{(k_{r+1})}\cdots x_{i_{n-1}}^{(k_{n-1})}(x_{i_{r}}\succ x_{i_n}).$$

Since, $(x_{i_{r}}\succ x_{i_n})\in \Der\Perm_{\succ}\<X\> $ the remaining part can be rewritten by inductive hypothesis.

Now, consider the second case such that there is no generator with derivation degree 1. That is,
$$x_{i_1}^{(k_1)}\cdots x_{i_{n-1}}^{(k_{n-1})}x_{i_n},$$
where $k_1,\ldots, k_{n-1}\neq 1$  and $k_1 +\ldots + k_{n-1}=n-1.$ Then we may assume that  
$$x_{i_1}^{(k_1)}\cdots x_{i_{n-1}}^{(k_{n-1})}x_{i_n}=x_{i_1}^{(k_1)}\cdots x_{i_{r-1}}^{(k_{r-1})}x_{i_r}^{(k_r)} x_{i_{r+1}}\cdots x_{i_{n-1}} x_{i_n},$$ where $k_1,\ldots, k_{r-1}\neq 1$ and $k_1 +\ldots + k_{r-1}=n-1.$
We choose any generator $x_{i_j}$ with minimal derivation degree greater than $1.$ Without  loss of generality, we assume that the generator $x_{i_r}$ has minimal derivation degree, that is $k_r=min(\{k_1,k_2,\ldots,k_r\}).$ Then we  rewrite it as follows:
\begin{multline*}
x_{i_1}^{(k_1)}\cdots x_{i_{r-1}}^{(k_{r-1})}x_{i_r}^{(k_r)} x_{i_{r+1}}\cdots x_{i_{n-1}} x_{i_n}= \\
x_{i_1}^{(k_1)}\cdots x_{i_{r-1}}^{(k_{r-1})}((\cdots(x_{i_r}\succ x_{i_{r+1}})\succ\cdots)\succ x_{j_{r+k_r}}) x_{j_{r+k_r+1}}\cdots x_{i_n} \\
-\sum_{t,t_1,\ldots, t_{k_r}<k_r} x_{i_1}^{(k_1)}\cdots x_{i_{n-1}}^{(k_{n-1})} x_{i_r}^{(t)} x_{j_{r+1}}^{(t_1)}\cdots x_{j_{r+k_r}}^{(t_{k_r})}x_{j_{r+k_r+1}}\cdots x_{i_n}.
\end{multline*}
The monomial
$$x_{i_1}^{(k_1)}\cdots x_{i_{r-1}}^{(k_{r-1})}((\cdots(x_{i_r}\succ x_{i_{r+1}})\succ\cdots)\succ x_{j_{r+k_r}}) x_{j_{r+k_r+1}}\cdots x_{i_n}$$
can be written in terms of $\succ$ by inductive hypothesis. For the monomials of the form
$$x_{i_1}^{(k_1)}\cdots x_{i_{n-1}}^{(k_{n-1})} x_{i_r}^{(t)} x_{j_{r+1}}^{(t_1)}\cdots x_{j_{r+k_r}}^{(t_{k_r})}x_{j_{r+k_r+1}}\cdots x_{i_n},$$ where $t,t_1,\ldots,t_{k_r}<k_r,$ we have 
if $t,t_1,\ldots,t_{k_r}>1,$ we use the same way which is given above several times and finally obtain the generator with derivation degree $1$ which correspond to the first case.
\end{proof}

Let us recall some facts about Novikov dialgebra(see \cite{KS2022}):
\begin{itemize}
  \item Every Novikov dialgebra can be embedded into perm algebra with derivation relative to the operations
$$x_1\vdash x_2 = x_1 d(x_2),\;\;\; x_1\dashv x_2 = d(x_2) x_1.$$
  \item If $a\in\Der\Perm\<X\>$ of the form (\ref{derpermmonom}) then $a$ can be written in terms of $\vdash$ and $\dashv$ if and only if $\wt(a)=-1$.
  \item If $a\in\Der\Perm\<X\>$ of the form (\ref{derpermmonom}) then $a$ can be written in terms of $\prec$ if and only if $\wt(a)=-1$ and $k_n\neq 0$.
\end{itemize}
Now, we are ready to prove the Theorem \ref{directsum}.
\begin{proof}
Firstly, we equalize the operations of Novikov dialgebra as follows:
$$x_1\succ x_2=x_2\dashv x_1,\;\;\; x_1\prec x_2=x_1\vdash x_2.$$
We denote by $P_{(-1)}$ the subspace of monomials $\Der\Perm\<X\>$ of the weight $-1$. Also, we denote by $P_{(-1,k_n= 0)}$ and $P_{(-1,k_n\neq 0)}$ the subspaces of monomials $P_{(-1)}$ of the form (\ref{derpermmonom}) with $k_n=0$ and $k_n\neq 0$, respectively.

By the facts that the free $\di$-$\Nov\<X\>$ isomorphic to $\Der\Perm\<X\>$ and a linear basis of $\di$-$\Nov\<X\>$ in $\Der\Perm\<X\>$ are monomials of the weight $-1$,  we can represent a linear basis of $\di$-$\Nov\<X\>$ as follows:
$$P_{(-1)}=P_{(-1,k_n= 0)}\oplus P_{(-1,k_n\neq 0)}.$$
By Theorem \ref{criterion} and the third fact, all monomials of subspaces $P_{(-1,k_n= 0)}$ and $P_{(-1,k_n\neq 0)}$ can be written in terms of $\succ$ and $\prec$ reprectively, i.e. in terms of $\dashv$ and $\vdash$.
\end{proof}

\section{Proof of Theorem \ref{Cohn}}

\begin{proof}
Let us prove it by Cohn's criterion on homomorphic images of special algebras \cite{Cohn}. Suppose that $\alpha$ is an ideal of $\Der\Perm_\succ\<X\>$ and $\{\alpha\}$ is an ideal of $\Der\Perm\<X\>$ generated by by the set $\alpha$. It is enough to prove that $\Der\Perm_\succ\<X\>\cap \{\alpha\}\subseteq\alpha$. 

Assume that $g_i$ $ (i\in I)$ are generators of the ideal $\alpha.$ Let $w$ be a non-zero element of $\Der\Perm_\succ\<X\>\cap \{\alpha\}$. Then 
$$w=\sum_k \mu_k w_k, \,\, (\mu_k\in \mathbb{K})$$ 
is a linear combination of  monomials $w_k$ in $X^{(w)}$ and $ g_i$ $ (i\in I)$  such that each monomial  is linear by at least one generator of $\alpha.$  

Let $w_k=a_1\cdots a_n$ be a term of $w$ in the linear combination. Assume that there is $r$ such that $a_r=g^{(k_r)}_i,$ where $g_{i}$ is a generator of the ideal $\alpha$.  Then express $a_1,\ldots,a_{r-1},a_{r+1},\ldots ,a_{n}$ in terms of $x^{(i_j)}_{i_j}\in X^{(w)}$, therefore we can assume that $a_1,\ldots,a_{r-1},a_{r+1},\ldots ,a_{n}\in X^{(w)}.$ 

Since $w\in\Der\Perm_\succ\<X\>$, by Theorem \ref{criterion} we have $\wt(w_k)=-1$ and  $a_n\in X$. That is 
$$w=\sum_k \mu_k w_k, \,\, (\mu_k\in \mathbb{K})$$
where
$$w_k=x_{i_1}^{(k_1)}\cdots g^{(k_r)}_i\cdots x_{i_{n-1}}^{(k_{n-1})}x_{i_n},$$
where $g_{i}$ is a generator of the ideal $\alpha$.  But by Theorem \ref{criterion} we can express each $w_k$ as linear combination of monomials in terms of product $\succ$ with generators $x_{i_1},\ldots, x_{i_n}$ and $g_i$ $ (i\in I),$ where $g_i$ appears in each monomial at least once.
Therefore, every $w_k\in \alpha,$ and we have $w\in \alpha.$ 




\end{proof}

\section{A linear basis of the free $\Der\Perm_{\succ}$ algebra}

A monomial basis of the free $\Der\Perm_{\prec}$ algebra was constructed in \cite{KS2022}. By Theorem \ref{directsum}, to construct a free basis of Novikov dialgebra it is enough to solve the same problem for $\Der\Perm_{\succ}$ algebra. In this section, we construct a monomial basis of the free $\Der\Perm_{\succ}$ algebra.
Let us define an order on generators of $\Der\Perm\<X\>$ as follows:
\begin{center}
$x_i^{(r_m)}>x_j^{(r_n)}$
\end{center}
if $r_m>r_n$, or $r_m=r_n$ and $i>j$.

Let us define a normal form for monomials of the form (\ref{derpermmonom}) of weight $-1$ and $k_n=0$ as follows:
\begin{equation}\label{normalmonomialderperm}
x_{k_1}'\ldots x_{k_m}'x_{j_1}^{(r_1)}\cdots x_{j_n}^{(r_n)}x_{i_l}\cdots x_{i_{2}} x_{i_{1}},
\end{equation}
where
$$r_1,\ldots, r_n\neq 1,\;\;\; x_{i_l}\geq \ldots \geq x_{i_{2}},$$
$$x_{j_n}^{(r_n)}\geq\ldots\geq x_{j_1}^{(r_1)},\;\;\;  k_m\geq\ldots\geq k_1.$$

Denote by $P(X)$ the set of all monomials of normal forms~\eqref{derpermmonom} and denote by $\Magma\langle X\rangle$ the free magma algebra with binary operation~$\succ$ generated by~$X$.

We define a mapping $\phi$ inductively on monomials of the form $(\ref{normalmonomialderperm})$ as follows:
\[
\phi:\Der\Perm\<X\>\rightarrow\Magma\<X\>,
\]
\[
\phi(x_{i_n}^{(n-1)}x_{i_{n-1}}\cdots x_{i_2}x_{i_1})=((\cdots(x_{i_n}\succ x_{i_{n-1}})\succ\cdots)\succ x_{i_2})\succ x_{i_1},
\]
where $x_{i_{n-1}}\geq\ldots\geq x_{i_2}$.
\begin{multline*}
    \phi(x_{k_1}'\ldots x_{k_m}'x_{j_1}^{(r_1)}\cdots x_{j_n}^{(r_n)}x_{i_l}\ldots x_{i_{2}} x_{i_{1}})=\\
\phi(x_{k_1}'\ldots x_{k_m}'x_{j_1}^{(r_1)}\cdots x_{j_{n-1}}^{(r_{n-1})}P_1x_{i_{l-r_n}}\ldots x_{i_{2}} x_{i_{1}}),
\end{multline*}
where $P_1=\phi(x_{j_n}^{(r_n)}x_{i_l}\ldots x_{i_{l+1-r_n}})$.
We may assume that $P_1$ is a new generator i.e. $X_1=X\cup \{P_1\}$ such that $\wt(P_1)=-1$ and $P_1\geq x_{i_{l-r_n}}\geq \ldots\geq x_1$. Calculating by the given rule, finally, we get
\begin{multline}\label{generalform}
\phi(x_{k_1}'\cdots x_{k_m}'x_{j_1}^{(r_1)}\cdots x_{j_n}^{(r_n)}x_{i_l}\cdots x_{i_{2}} x_{i_{1}})=\cdots\\
=x_{k_1}\succ(\cdots(x_{k_m}\succ((\cdots(x_{j_1}\succ P_{n-1})\cdots\succ x_{i_2})\succ x_{i_1}))\cdots).
\end{multline}

We denote by $[a]$ the image~$\phi(a)$.
Define a map
$\tau\colon \Magma\langle X\rangle\to \Der\Perm\langle X\rangle$ by the formula
$\tau(x) = x$, $x\in X$, and
$\tau(a\succ b) = a'b$.
For example, if $x,y,z\in X$, then
$\tau((x\succ y)\succ z) = (x'y)'z = x''yz + x'y'z$.

Let us define a sequence on a monomial of the set $P(X)$ of the form (\ref{normalmonomialderperm}) as follows:
$$S(a)=(\underset{m}{\underbrace{1,\ldots ,1}},r_1,\ldots,r_{n-1},r_n),$$
i.e. that is a sequence of positive numbers in a non-decreasing order.
Define an order on that sequences lexicographically.
Define an order on monomials of the set $P(X)$ as follows:
\[
a>b\quad \textrm{if}\quad S(a)>S(b)
\]
and call that order by $der$-$lex$.

\begin{lemma}\label{partitionderivation}
Let $a = x_{k_1}'\cdots x_{k_m}'x_{j_1}^{(r_1)}\cdots x_{j_n}^{(r_n)}x_{i_l}\cdots x_{i_{2}} x_{i_{1}}$ be from $P(X)$.
Then $\tau([a]) = a + \sum_j b_j$, where $a>b_j$ for all $j$.
\end{lemma}

\begin{proof}
By definition of the mapping $\tau$, it is enough to prove it for
$$a=x_{j_1}^{(r_1)}\cdots x_{j_n}^{(r_n)}x_{i_l}\ldots x_{i_{2}} x_{i_{1}}.$$
By Leibniz rule for derivation, we have
\begin{multline*}
\tau([x_{j_1}^{(r_1)}\cdots x_{j_n}^{(r_n)}x_{i_l}\cdots x_{i_{2}} x_{i_{1}}])=\tau((\cdots(x_{j_1}\succ P_{n-1})\cdots\succ x_{i_2})\succ x_{i_1})=\\
x_{j_1}^{(r_1)} \tau(P_{n-1}) x_{i_l}\cdots x_{i_2} x_{i_1}+\sum_{p<r_1} x_{j_1}^{(p)} \tau(P_{n-1}) \cdots x_{i_1}.
\end{multline*}
By der-lex order, $a$ is a bigger that all monomials of $\sum_{p<r_1} x_{j_1}^{(p)} \tau(P_{n-1}) \cdots x_{i_1}$.
In the same way, applying Leibniz rule to $\tau(P_{n-1})$ we obtain the same result as above. Finally, we obtain only one monomial which is not less than $a$. That is $a$ itself.
\end{proof}

Under the map $\phi$ for each monomial of $P(X)$ correspond a unique monomial of magma algebra. Let us denote by $[P(X)]$ the set of all such monomials obtained under the mapping $\phi$. Also, we define a linear mapping $\tau$ as above for the following algebras:
$$\tau:\Der\Perm_{\succ}\rightarrow \Der\Perm.$$

\begin{theorem}\label{newbasenov}
The set $[P(X)]$ is a linear basis of $\Der\Perm_\succ\<X\>$.
\end{theorem}

\begin{proof}
It is enough to show that for every
weight-homogeneous differen\-tial polynomial
$a\in \Der\Perm\<X\>$ of the form (\ref{normalmonomialderperm}) there exists $[p_1],\ldots,[p_r]\in [P(X)]$ such that
\begin{equation}\label{equality}
a=\sum_{i=1}^{r} \lambda_i \tau([p_i]), \quad \lambda_i\in \Bbbk .
\end{equation}
Indeed, assume \eqref{equality} holds and
$w\in \Der\Perm_{\succ}\langle X\rangle $
is a homogeneous element of degree $n$.
Then $a=\tau(w)$ is a weight-homogeneous differential polynomial
which can be presented by \eqref{equality}. Since $\tau $ is injective, we obtain
\[
w=\sum_{i=1}^{r} \lambda_i [v_{i}], \quad [v_i]\in [P_n(X)].
\]
Hence, the set $[P_n(X)]$ is linearly complete in the space $F_n$
of degree $n$ elements in $ \Der\Perm_{\succ}\langle X\rangle$.
Thus this is a basis since $|P_n(X)|=|[P_n(X)]|=\dim F_n$.

The statement \eqref{equality} is enough to prove for monomials only. Let us prove it by induction on the degree of $a$. The base of induction is $n=2$. Suppose that the statement holds for the degree up to $n-1$. Let us prove that statement for monomial of degree $n$.

Case 1: if
$$x_{k_1}'\ldots x_{k_m}'x_{j_1}^{(r_1)}\cdots x_{j_n}^{(r_n)}x_{i_l}\cdots x_{i_{2}} x_{i_{1}},
,$$
such that $m\neq 0$, then by inductive hypothesis for the subword of $a$, we have
$$x_{j_1}^{(r_1)}\cdots x_{j_n}^{(r_n)}x_{i_l}\cdots x_{i_{2}} x_{i_{1}}=\sum_i \lambda_i\tau([p_i])$$
and
$$a=\sum_i \lambda_i\tau(x_{k_1}\succ(\ldots\succ(x_{k_m}\succ([p_i]))\ldots)),$$
where $x_{k_1}\succ(\ldots\succ(x_{k_m}\succ([p_i]))\ldots)\in [P(X)]$.

Case 2: It is enough to consider only monomials of the form
\begin{equation}\label{good}
a=x_{j_1}^{(r_1)}\cdots x_{j_n}^{(r_n)}x_{i_l}\cdots x_{i_{2}} x_{i_{1}}, \end{equation}
such that $r_t\neq 1$ for all $t$.

For $a$, we have the following equality:
\begin{equation}\label{representation}
a=\tau([a])-\sum_i b_i
\end{equation}
and by Lemma \ref{partitionderivation}, $a>b_i$ for all $i$. Applying the equality (\ref{representation}) to $b_i$ several times, finally, by der-lex order we obtain a generator with derivation degree $1$ which correspond to the first case.
\end{proof}

\section{The defining identities of variety $\Der\Perm_{\succ}$}

\begin{definition}
An algebra with identities $(\ref{id1})$, $(\ref{id2})$ and $(\ref{id3})$ is called $\SLC$-algebra(special left-commutative).
\end{definition}

Let the set $N(X)$ be a linear basis of the free Novikov algebra. Every non-associative word $u$ in $X$ can be written as 
\begin{equation}\label{LNform}
u = L(u_1,\dots, u_k, x):=u_{1}\succ (u_{2} \succ ( \dots \succ (u_{k}\succ x)\dots )),
\end{equation}
where $u_{j}$ are non-associative words, $x\in X$.
Let $k$ be the {\em length\/} of $u$.
For a linear combination of \eqref{LNform}, the length 
is the maximal length of its summands.

\begin{lemma}\label{lem:Lemma2}
Every element of $\SLC\langle X\rangle $
may be written as a linear combination 
of words of the form \eqref{LNform} with $u_j\in N(X)$, $x\in X$.
\end{lemma}

\begin{proof}
It follows from identities (\ref{id1}) and (\ref{id3}).
\end{proof}

Let us recall a free basis of Novikov algebra \cite{Gub-Sar}: that is the monomials of the from (\ref{generalform}) with additional condition $i_2\geq i_1$. By (\ref{generalform}), $u_i$ has a form
$$x_{k_1}\succ(\cdots(x_{k_m}\succ((\cdots(x_{j_1}\succ P_{n-1})\cdots\succ x_{i_2})\succ x_{i_1}))\cdots).$$

Now, we are ready to prove Theorem \ref{specialidentities}.

\begin{proof}
To prove that the defining identities of the variety $\Der\Perm_{\succ}$ are $(\ref{id1})$, $(\ref{id2})$ and $(\ref{id3})$ we have to show that any monomial of the free $\SLC$-algebra can be written as a sum of monomials from the set $[P(X)]$. In the other words, it means that $\SLC\<X\>\cong\Der\Perm_{\succ}\<X\>$. Let us prove it by induction on the degree of monomial
$$u=u_{1}\succ (u_{2} \succ ( \cdots \succ (u_{k}\succ x)\cdots )),$$
where $u_i\in N(X)$.
The base of the induction is the case when the degree of $u$ is $4$. For monomial of degree $n$, we consider several cases:

Case 1: $u_i=x_j\in X$. For that case, by (\ref{id1}) we have
\begin{multline*}
    u=u_{1}\succ (u_{2} \succ ( \cdots \succ (u_{k}\succ x)\cdots ))=x_j\succ(u_{1}\succ ( \cdots \succ (u_{k}\succ x)\cdots ))
\end{multline*}
and the monomial
$$u_{1}\succ ( \cdots\succ(u_{i-1}\succ(u_{i+1}\succ\cdots \succ(u_{k}\succ x)\cdots))\cdots )$$
can be rewritten by inductive hypothesis.

Case 2: $u_i\neq x_j$. By (\ref{id1}), without lose of generality we can assume that $u_1=x_{k_1}\succ(\cdots(x_{k_m}\succ((\cdots(x_{j_1}\succ P_{n-1})\cdots\succ x_{i_2})\succ x_{i_1}))\cdots)$ and we have
$$u=(x_{k_1}\succ(\cdots(x_{k_m}\succ((\cdots(x_{j_1}\succ P_{n-1})\cdots\succ x_{i_2})\succ x_{i_1}))\cdots))\succ R,$$
where $m\geq2$ and $R=u_{2} \succ ( \cdots \succ (u_{k}\succ x)\cdots )$. By (\ref{id2}),
\begin{multline*}
   u= x_{k_1}\succ((\cdots(x_{k_m}\succ((\cdots(x_{j_1}\succ P_{n-1})\cdots\succ x_{i_2})\succ x_{i_1}))\cdots)\succ R) \\
   +x_{k_2}\succ((x_{k_1}\succ(\cdots(x_{k_m}\succ((\cdots(x_{j_1}\succ P_{n-1})\cdots\succ x_{i_2})\succ x_{i_1}))\cdots))\succ R) \\
   x_{k_1}\succ(x_{k_2}\succ((\cdots(x_{k_m}\succ((\cdots(x_{j_1}\succ P_{n-1})\cdots\succ x_{i_2})\succ x_{i_1}))\cdots)\succ R))
\end{multline*}
which correspond to the first case.



Case 3. $u_i\neq x_j$, and
\begin{equation}\label{goodnov1}
u_i=(\cdots((x_{j_1}\succ P_{n-1})\succ x_{i_p})\cdots\succ x_{i_2})\succ x_{i_1}
\end{equation}
or
\begin{equation}\label{goodnov2}
u_i=x_{k_1}\succ((\cdots((x_{j_1}\succ P_{n-1})\succ x_{i_p})\cdots\succ x_{i_2})\succ x_{i_1}).
\end{equation}

The identity (\ref{id2}) can be represented as follows:
\begin{multline}\label{id2deg5}
    (x_1\succ x_2)\succ((x_3\succ x_4)\succ x_5)=((x_1\succ x_2)\succ(x_3\succ x_4))\succ x_5\\
    -((x_1\succ x_2)\succ x_4)\succ(x_3\succ x_5)+x_4\succ ((x_1\succ x_2)\succ(x_3\succ x_5)).
\end{multline}
For
$$u = L(u_1,\dots, u_k, x)=u_1\succ (u_2\succ(R)),$$ suppose that 
$u_1=(\cdots(x_{p_1}\succ P_{n-1})\cdots\succ x_{i_2})\succ x_{i_1}$,
$u_2=(\cdots(x_{t_1}\succ P_{m-1})\cdots\succ x_{j_2})\succ x_{j_1}$ and without lose of generality we assume that the degree of $u_1$ bigger or equal than degree of $u_2$.
We use (\ref{id2deg5}) on $u$ and obtain
\begin{multline*}
    u_1\succ(u_2\succ R)= (u_1\succ u_2)\succ R  \\
    - (u_1\succ x_{j_1})\succ(\overline{u_{2}} \succ R) + x_{j_1}\succ (u_1\succ(\overline{u_{2}} \succ R)),
\end{multline*}
where $\overline{u_{2}}=(\cdots(x_{t_1}\succ P_{m-1})\cdots)\succ x_{j_2}$.
The monomial $x_{j_1}\succ (u_1\succ(\overline{u_{2}} \succ R))$ correspond to the first case. For the monomial $(u_1\succ x_{j_1})\succ(\overline{u_{2}} \succ R)$ we use (\ref{id2deg5}) in the same way several times until $\overline{u_2}$ comes to the form $x_{t_1}\succ P_{m-1}$. Finally, we obtain
$$\overline{u_1}\succ((x_{t_1}\succ P_{m-1})\succ R)$$
and by (\ref{id2deg5}) we have
\begin{multline*}
    \overline{u_1}\succ((x_{t_1}\succ P_{m-1})\succ R)=(\overline{u_1}\succ(x_{t_1}\succ P_{m-1}))\succ R \\
    -(\overline{u_1}\succ P_{m-1})\succ (x_{t_1}\succ R)+x_{t_1}\succ (P_{m-1}\succ (\overline{u_1}\succ R))=(\overline{u_1}\succ(x_{t_1}\succ P_{m-1}))\succ R \\
    -x_{t_1}\succ ((\overline{u_1}\succ P_{m-1})\succ R)+x_{t_1}\succ (P_{m-1}\succ (\overline{u_1}\succ R)).
\end{multline*}
The monomials $x_{t_1}\succ ((\overline{u_1}\succ B)\succ R)$ and $x_{t_1}\succ (P_{m-1}\succ (\overline{u_1}\succ R))$ correspond to the first case. Notice that for the remaining monomials $(u_1\succ u_2)\succ R$ and $(\overline{u_1}\succ(x_{t_1}\succ P_{m-1}))\succ R$ the number of left adjoint operators decreased to one i.e. each monomial has the form
$$L(v_1,u_3,\ldots,u_k,x)\;\;\;\textrm{and}\;\;\;L(v_2,u_3,\ldots,u_k,x),$$
respectively, where $v_1=u_1\succ u_2$ and $v_2=\overline{u_1}\succ(x_{t_1}\succ P_{m-1})$.

In the same way we consider the case that
$$u_1=x_{k_1}\succ((\cdots(x_{p_1}\succ P_{n-1})\cdots\succ x_{i_2})\succ x_{i_1})$$
and $u_2$ has arbitrary form. By (\ref{id2deg5}), we have
\begin{multline*}
u_1\succ (u_2\succ R)=u_2\succ (u_1\succ R)=
    (u_2\succ u_1)\succ R  \\
    - (u_2\succ \overline{u_{1}})\succ(x_{k_1} \succ R) + \overline{u_{1}}\succ (u_2\succ(x_{k_1} \succ R))=\\
    (u_2\succ u_1)\succ R
    - x_{k_1}\succ((u_2\succ \overline{u_{1}}) \succ R) + x_{k_1}\succ (u_2\succ(\overline{u_{1}} \succ R)),
\end{multline*}
where the second and the third monomials correspond to the first case and the number of left adjoint operator of monomial $(u_2\succ u_1)\succ R$ decreased to one.

Applying the above method several times, we end up with monomials of the form
$$L(x_{s_1},\cdots,x_{s_r},N,x_j)=x_{s_1}\succ(\cdots (x_{s_r}\succ(N\succ x_j))\cdots),$$
where $N$ is a Novikov word of the form (\ref{goodnov1}) and (\ref{goodnov2}).
We obtain the following result:
\begin{lemma}
Every monomial of $\SLC\<X\>$ algebra can be written as a sum of monomials of the following form:
$$x_{s_1}\succ(\cdots (x_{s_r}\succ(N\succ x_j))\cdots),$$
where $N$ is a Novikov word of the form (\ref{goodnov1}) and (\ref{goodnov2}).
\end{lemma}

Now, we consider the monomials of $\SLC\<X\>$ of the form
$$u=u_1\succ x.$$
Let us prove that
$$u=\sum_i \alpha_i p_i,\;\;\; p_i\in [P(X)].$$

Firstly, we consider the case that $u_1$ has the form $(\ref{goodnov2})$. By (\ref{normalmonomialderperm}) and (\ref{generalform}), $u\notin P(X)$ if and only if $p=1$ and $k_1>j_1$:
$$u=(x_{k_1}\succ ((x_{j_1}\succ P_{n-1})\succ x_{i_1}))\succ x\notin [P(X)].$$
Notice that the identity 
\begin{multline*}
    (x_{k_1}\succ ((x_{j_1}\succ P_{n-1})\succ x_{i_1}))\succ x=
    (x_{j_1}\succ ((x_{k_1}\succ P_{n-1})\succ x_{i_1}))\succ x-\\
    x_{j_1}\succ (((x_{k_1}\succ P_{n-1})\succ x_{i_1})\succ x) 
    -x_{k_1}\succ ((x_{j_1}\succ (P_{n-1}\succ x_{i_1}))\succ x) \\
    +x_{k_1}\succ (x_{j_1}\succ ((P_{n-1}\succ x_{i_1})\succ x))
    +x_{k_1}\succ (((x_{j_1}\succ P_{n-1})\succ x_{i_1})\succ x) \\
    +x_{j_1}\succ ((x_{k_1}\succ (P_{n-1}\succ x_{i_1}))\succ x)
    -x_{j_1}\succ (x_{k_1}\succ ((P_{n-1}\succ x_{i_1})\succ x))
\end{multline*}
holds in $\SLC\<X\>$ algebra. It can be proved using computer algebra as software programs  Albert \cite{Albert}. The first monomial of the right side belongs to $[P(X)]$ and the other monomials correspond to the first case. So, $u$ can be written as a sum of monomials of $[P(X)]$.


It remain to consider the case that $u_1$ has the form $(\ref{goodnov1})$. By (\ref{normalmonomialderperm}) and (\ref{generalform}), we may assume that
$$u_1=(\cdots((x_{j}\succ P_{n-1})\succ x_{i_r}) \cdots\succ x_{i_2})\succ x_{i_1}$$
and
$$P_{n-1}=(\cdots((x_{t}\succ P_{n-2})\succ x_{l_k})\cdots\succ x_{l_2})\succ x_{l_1},$$
where $k>r$, or $k=r$ and $t>j$. If $k>r+2$, or $k=r+1$ and $t>j$ then
$$u_1\succ x\in [P(X)].$$
To save space we omit the multiplication $\succ$.
We have only two cases for which
$$u_1\succ x\notin [P(X)].$$
There are $k=r$, or $k=r+1$ and $t<j$. Let us start the second induction on $r$ for the monomial $u_1$.
The base of induction is 
$r=0$ and  this case is considered above i.e.
$$u=u_1\succ x=(x_j\succ ((\cdots((x_{t}\succ P_{n-2})\succ x_{l_k})\cdots\succ x_{l_2})\succ x_{l_1}))\succ x.$$
The $``\equiv"$ symbol means that we are excluding all monomials for which $r$ is smaller than given.


Now, we are ready to consider the general case of $u$.
By (\ref{id2}) we rewrite the monomial $u$ as follows:
\begin{multline*}
    u_1 x=((\cdots((x_{j} P_{n-1}) x_{i_r}) \cdots x_{i_2}) x_{i_1}) x\equiv  \\
    (((\cdots(x_{j} x_{i_r}) \cdots x_{i_2}) x_{i_1}) P_{n-1}) x=  \\
    ((\cdots(x_{j} x_{i_r}) \cdots x_{i_2}) x_{i_1}) (((\cdots((x_{t} P_{n-2}) x_{l_k})\cdots x_{l_2}) x_{l_1}) x)+  \\
    (\cdots((x_{t} P_{n-2}) x_{l_k})\cdots x_{l_2}) ((((\cdots(x_{j} x_{i_r}) \cdots x_{i_2}) x_{i_1}) x_{l_1}) x)+  \\
    x_{l_1}((\cdots((x_{t} P_{n-2}) x_{l_k})\cdots x_{l_2})(((\cdots(x_{j} x_{i_r}) \cdots x_{i_2}) x_{i_1}) x)).
\end{multline*}
The last term correspond to the first case. For the first and second terms we use (\ref{id1}) and (\ref{id2deg5}), and obtain
\begin{multline*}
    ((\cdots((x_{t} P_{n-2}) x_{l_k})\cdots x_{l_2}) x_{l_1}) (((\cdots(x_{j} x_{i_r}) \cdots x_{i_2}) x_{i_1}) x)+  \\
    (((\cdots(x_{j} x_{i_r}) \cdots x_{i_2}) x_{i_1}) x_{l_1}) ((\cdots((x_{t} P_{n-2}) x_{l_k})\cdots x_{l_2}) x)=  \\
    (((\cdots((x_{t} P_{n-2}) x_{l_k})\cdots x_{l_2}) x_{l_1}) ((\cdots(x_{j} x_{i_r}) \cdots x_{i_2}) x_{i_1})) x-  \\
    (\cdots(x_{j} x_{i_r}) \cdots x_{i_2})
    ((((\cdots((x_{t} P_{n-2}) x_{l_k})\cdots x_{l_2}) x_{l_1}) x_{i_1}) x) \\
    +x_{i_1}((\cdots(x_{j} x_{i_r}) \cdots x_{i_2})(((\cdots((x_{t} P_{n-2}) x_{l_k})\cdots x_{l_2}) x_{l_1})x))+ \\
    ((((\cdots(x_{j} x_{i_r}) \cdots x_{i_2}) x_{i_1}) x_{l_1}) (\cdots((x_{t} P_{n-2}) x_{l_k})\cdots x_{l_2})) x- \\
    (\cdots((x_{t} P_{n-2}) x_{l_k})\cdots x_{l_3}) (((((\cdots(x_{j} x_{i_r}) \cdots x_{i_2}) x_{i_1}) x_{l_1})x_{l_2}) x)+  \\
    x_{l_2}((((\cdots(x_{j} x_{i_r}) \cdots x_{i_2}) x_{i_1}) x_{l_1})((\cdots((x_{t} P_{n-2}) x_{l_k})\cdots x_{l_3}) x)).
\end{multline*}
The third and sixth terms correspond to the first case. The first term we rewrite as follows: 
\begin{multline*}
    (((\cdots((x_{t} P_{n-2}) x_{l_k})\cdots x_{l_2}) x_{l_1}) ((\cdots(x_{j} x_{i_r}) \cdots x_{i_2}) x_{i_1})) x\equiv \\
    ((\cdots(x_{j} x_{i_r}) \cdots x_{i_2}) (((\cdots((x_{t} P_{n-2}) x_{l_k})\cdots x_{l_2}) x_{l_1}) x_{i_1})) x\equiv \\
    ((\cdots((x_{j}(((\cdots((x_{t} P_{n-2}) x_{l_k})\cdots x_{l_2}) x_{l_1}) x_{i_1})) x_{i_r}) \cdots x_{i_2})) x \in [P(X)].
\end{multline*}
For the second, forth and fifth terms we use the same rewriting method as given above and finally obtain representation in terms of $[P(X)]$.
\end{proof}

\end{document}